\documentclass[draft]{amsart}
\usepackage{amssymb}
\usepackage{enumerate}
\makeatletter%
\def\th@mytheorem{%
  \let\thm@indent\noindent
  \thm@headfont{\bfseries}
    \itshape
}
\def\th@myremark{%
  \let\thm@indent\noindent
  \thm@headfont{\bfseries}
}
\def\TagsOnLeft{\tagsleft@true}
\def\TagsOnRight{\tagsleft@false}
\def\@makefnmark{\hbox{$\left(^{\@thefnmark}\right)$\;}}
\makeatother%
\theoremstyle{mytheorem}
\newtheorem{Theorem}{Theorem}[section]
\newtheorem{Corollary}[Theorem]{Corollary}

\newtheorem{Proposition}[Theorem]{Proposition}
\theoremstyle{myremark}
\newtheorem{Definition}[Theorem]{Definition}
\newtheorem{Remark}[Theorem]{Remark}
\def\bs{\boldsymbol}
\def\BB{\mathbb B}
\def\bu{$\bullet$\quad}
\def\BV{\bs{V}}
\def\BVV#1{\BV_{\!#1}}
\def\BX{\bs{X}}
\def\CC{\mathbb C}
\def\conv{\operatorname{conv}}
\def\CVX{\mathcal{CVX}}
\def\Dqed{\text{\popQED}\tag*{\qed}}
\def\halfskip{\vskip 6pt plus 1pt minus 1pt}
\def\intt{\operatorname{int}}
\def\NN{\mathbb N}
\def\PSH{\mathcal{PSH}}
\def\RR{\mathbb R}
\def\ttimes{\times\dots\times}
\def\too{\longrightarrow}
\def\tuu{\longmapsto}
\def\wdht{\widehat}
\def\wdtl{\widetilde}

\def\eps{\varepsilon}
\def\Omega{\varOmega}
\let\phi=\varphi
\def\Phi{\varPhi}

\TagsOnRight
\begin{document}
\title
{An elementary proof of the cross theorem in the Reinhardt case}

\author{Marek Jarnicki}
\address{Jagiellonian University, Institute of Mathematics,
{\L}ojasiewicza 6, 30-348 Krak\'ow, Poland} \email{Marek.Jarnicki@im.uj.edu.pl}

\author{Peter Pflug}
\address{Carl von Ossietzky Universit\"at Oldenburg, Institut f\"ur Mathematik,
Postfach 2503, D-26111 Oldenburg, Germany}
\email{pflug@mathematik.uni-oldenburg.de}
\thanks{The research was partially supported by the DFG-grant 436POL113/103//0-2.}



\begin{abstract}
We present an elementary proof of the cross theorem in the case of Reinhardt domains.
The results illustrates the well-known interrelations between the holomorphic geometry
of a Reinhardt domain and the convex geometry of its logarithmic image.
\end{abstract}

\maketitle


\section{Introduction. Main result.}
The problem of continuation of separately holomorphic functions defined on
a cross has been investigated in several papers, e.g.~\cite{Ber1912},  \cite{Sic1969a},\linebreak
\cite{Sic1969b}, \cite{AkhRon1973}, \cite{Zah1976}, \cite{Sic1981a}, \cite{Shi1989},
\cite{NguSic1991},\linebreak \cite{NguZer1991}, \cite{NguZer1995},
\cite{NTV1997}, \cite{AleZer2001}, \cite{Zer2002} and may be formulated in the form of
the following \textit{cross theorem}.

\begin{Theorem}\label{ThmMainCross}
Let $D_j\subset\CC^{n_j}$ be a domain of holomorphy and let $A_j\subset D_j$ be a locally pluriregular set,
$j=1,\dots,N$, $N\geq2$. Define the {\rm cross}
$$
\BX:=\bigcup_{j=1}^NA_1\ttimes A_{j-1}\times D_j\times A_{j+1}\ttimes A_N.
$$
Let $f:\BX\too\CC$ be {\rm separately holomorphic}, i.e.~for any
$(a_1,\dots,a_N)\in A_1\ttimes A_N$ and $j\in\{1,\dots,N\}$ the function
$$
D_j\ni z_j\tuu f(a_1,\dots,a_{j-1},z_j,a_{j+1},\dots,a_N)\in\CC
$$
is holomorphic. Then $f$ extends holomorphically to a uniquely determined function $\wdht f$ on
the domain of holomorphy
\begin{gather*}
\wdht{\BX}:=\Big\{(z_1,\dots,z_N)\in D_1\ttimes D_N: \sum_{j=1}^Nh_{A_j,D_j}^\ast(z_j)<1\Big\},
\tag{*}
\end{gather*}
where $h_{A_j,D_j}^\ast$ is the upper regularization of the {\rm relative extremal function}
$h_{A_j,D_j}$, $j=1,\dots,N$.
\end{Theorem}

Recall that $h_{A,D}:=\sup\{u\in\PSH(D): u\leq1,\;u|_A\leq0\}$.

Observe that in the case where $A_j$ is open, $j=1,\dots,N$, the cross $\BX$ is a domain
in $\CC^n$ with $n:=n_1+\dots+n_N$. Moreover, by the classical Hartogs lemma, every separately
holomorphic function on $\BX$ is simply holomorphic. Consequently, the formula (*) is nothing else as a
description of the envelope of holomorphy of $\BX$. Thus, it is natural to conjecture that
in this case the formula (*) may be obtained without the cross theorem machinery.
Unfortunately, we do not know any such a simplification.

\textit{The aim of this note is to present an elementary geometric
proof of Theorem \ref{ThmMainCross} in the case
where $D_j$ is a Reinhardt domain and $A_j$ is a non-empty Reinhardt open set, $j=1,\dots,N$.}
The proof (\S\;\ref{SectionProof}) will be based on well-known interrelations
between the holomorphic geometry of a Reinhardt domain and the convex geometry of its logarithmic
image. Moreover, the cross theorem for the Reinhardt case may be taught in any lecture on Several
Complex Variables; its proof needs only some basic facts for Reinhardt domains
(see \cite{JarPfl2008}).

\section{Convex geometry.}\label{SectionCG}

We begin with some elementary results related to the convex domains in $\RR^n$.

\begin{Definition}\label{DefCEF}
Let $\varnothing\neq S\subset U\subset\RR^n$, where $U$ is a convex domain. Define the
\textit{convex extremal function}
\index{convex extremal function}%
$$
\Phi_{S,U}:=\sup\{\phi\in\CVX(U),\; \phi\leq1,\;\phi|_S\leq0\},
$$
where $\CVX(U)$ stands for the family of all convex functions $\phi:U\too[-\infty,+\infty)$.
\end{Definition}

\begin{Remark}\label{RemCEF}
\begin{enumerate}[(a)]
\item\label{RemCEFa}
$\Phi_{S,U}\in\CVX(U)$, $0\leq\Phi_{S,U}<1$, and $\Phi_{S,U}=0$ on $S$.

\item\label{RemCEFb} $\Phi_{\conv(S),U}\equiv\Phi_{S,U}$.

\item\label{RemCEFe} If $\varnothing\neq S_k\subset U_k\subset\RR^n$, $U_k$ is a convex domain,
$k\in\NN$,
$S_k\nearrow S$, and $U_k\nearrow U$, then $\Phi_{S_k,U_k}\searrow\Phi_{S,U}$.

\item\label{RemCEFc} For $0<\mu<1$, let $U_\mu:=\{x\in U: \Phi_{S,U}(x)<\mu\}$
(observe that $U_\mu$ is a convex domain with $S\subset U_\mu$). Then $\Phi_{S,U_\mu}=(1/\mu)\Phi_{S,U}$
on $U_\mu$.

Indeed, the inequality ``$\geq$'' is obvious. To prove the opposite inequality, let
$$
\phi:=\begin{cases}\max\{\Phi_{S,U},\;\mu\Phi_{S,U_\mu}\} & \text{ on } U_\mu\\
\Phi_{S,U} & \text{ on } U\setminus U_\mu
\end{cases}.
$$
Then $\phi\in\CVX(U)$, $\phi<1$, and $\phi=0$ on $S$. Thus $\phi\leq\Phi_{S,U}$ and
hence $\Phi_{S,U_\mu}\leq(1/\mu)\Phi_{S,U}$ in  $U_\mu$.

\item\label{RemCEFd}
Let $\varnothing\neq S_j\subset U_j\subset\RR^{n_j}$, where $U_j$ is a convex domain,
$j=1,\dots,N$, $N\geq2$. Put
$$
W:=\Big\{(x_1,\dots,x_N)\in U_1\ttimes U_N:\sum_{j=1}^N\Phi_{S_j,U_j}(x_j)<1\Big\}
$$
(observe that $W$ is a convex domain with $S_1\ttimes S_N\subset W$). Then
$$
\Phi_{S_1\ttimes S_N,W}(x)=\sum_{j=1}^N\Phi_{S_j,U_j}(x_j),\quad x=(x_1,\dots,x_N)\in W.
$$

Indeed, the inequality ``$\geq$'' is obvious.
To prove the opposite inequality we use induction on $N\geq2$.

Let $N=2$. To simplify notation write $A:=S_1$, $U:=U_1$, $B:=S_2$, $V:=U_2$.
Observe that $T:=(A\times V)\cup(U\times B)\subset W$ and directly from the
definition we get
$$
\Phi_{A\times B,W}(x,y)\leq\Phi_{A,U}(x)+\Phi_{B,V}(y),\quad (x,y)\in T.
$$
Fix a point $(x_0,y_0)\in W\setminus T$. Let
\begin{gather*}
\mu:=1-\Phi_{A,U}(x_0)\in(0,1],\quad
V_\mu:=\{y\in V: \Phi_{B,V}(y)<\mu\},\\
\phi:=\frac1\mu(\Phi_{A\times B,W}(x_0,\cdot)-\Phi_{A,U}(x_0)).
\end{gather*}
Then $\phi$ is a well-defined convex function on $V_\mu$, $\phi<1$ on $V_\mu$, and
$\phi\leq0$ on $B$. Thus, by (\ref{RemCEFc}),
$\phi(y_0)\leq\Phi_{B,V_\mu}(y_0)=\frac1\mu\Phi_{B,V}(y_0)$, which finishes the proof.

\halfskip

Now, assume that the formula is true for $N-1\geq2$. Put $S':=S_1\ttimes S_{N-1}$,
$$
W':=\{(x_1,\dots,x_{N-1})\in U_1\ttimes U_{N-1}:\sum_{j=1}^{N-1}\Phi_{S_j,U_j}(x_j)<1\}.
$$
Then, by the inductive hypothesis, we have
$$
\Phi_{S',W'}(x')=\sum_{j=1}^{N-1}\Phi_{S_j,U_j}(x_j),\quad x'=(x_1,\dots,x_{N-1})\in W'.
$$
Consequently,
$$
W=\{(x',x_N)\in W'\times U_N: \Phi_{S',W'}(x')+\Phi_{S_N,U_N}(x_N)<1\}.
$$
Hence, using the case $N=2$ (to $S'\subset W'$ and $S_N\subset U_N$), we get
\begin{multline*}
\Phi_{S_1\ttimes S_N,W}(x)=\Phi_{S',W'}(x')+\Phi_{S_N,U_N}(x_N)=
\sum_{j=1}^N\Phi_{S_j,U_j}(x_j),\\
x=(x',x_N)=(x_1,\dots,x_N)\in W.
\end{multline*}
\end{enumerate}
\end{Remark}

Notice that properties (\ref{RemCEFc}) and (\ref{RemCEFd})
correspond to analogous properties of
the relative extremal function --- cf.~e.g.~\cite{Sic1981a}.

\begin{Proposition}\label{PropCEF}
Let $\varnothing\neq S_j\subset U_j\subset\RR^{n_j}$, where
$U_j$ is a convex domain and $\intt S_j\neq\varnothing$, $j=1,\dots,N$, $N\geq2$,
and define the {\rm cross}
$$
T:=\bigcup_{j=1}^N S_1\ttimes S_{j-1}\times U_j\times S_{j+1}\ttimes S_N.
$$
Then
\index{cross}%
\begin{gather*}
\conv(T)=\Big\{(x_1,\dots,x_N)\in U_1\ttimes U_N: \sum_{j=1}^N\Phi_{S_j,U_j}(x_j)<1\Big\}=:W.
\quad\footnotemark
\end{gather*}
\footnotetext{It seems to us that this ``convex cross theorem'' is so far nowhere in the literature.}
\end{Proposition}

\begin{proof} We may assume that $S_j$ is convex, $j=1,\dots,N$
(cf.~Remark \ref{RemCEF}(\ref{RemCEFb})).
The inclusion ``$\subset$'' is obvious. Let
\begin{gather*}
T_j:=S_1\ttimes S_{j-1}\times U_j\times S_{j+1}\ttimes S_N,\quad j=1,\dots,N,\\
T':=\bigcup_{j=1}^{N-1} S_1\ttimes S_{j-1}\times U_j\times S_{j+1}\ttimes S_{N-1},\quad
S':=S_1\ttimes S_{N-1}.
\end{gather*}
Recall (cf.~\cite{Roc1972}, Theorem 3.3) that
\begin{multline*}
\conv(T)\\
=\bigcup_{\substack{t_1,\dots,t_N\geq0 \\ t_1+\dots+t_N=1}}t_1T_1+\dots+t_NT_N=
\conv((\conv(T')\times S_N)\cup(S'\times U_N)).\tag{**}
\end{multline*}

We use induction on $N$.

$N=2$:
To simplify notation write $A:=S_1$, $U:=U_1$, $p:=n_1$, $B:=S_2$, $V:=U_2$, $q:=n_2$.
Using Remark \ref{RemCEF}(\ref{RemCEFe}), we may assume that $U, V$ are bounded.

Since $\conv(T)$ is open and $\conv(T)\subset W$,
we only need to show that for every $(x_0,y_0)\in\partial(\conv(T))\cap(U\times V)$ we have
$\Phi_{A,U}(x_0)+\Phi_{B,V}(y_0)=1$. Since $U, V$ are bounded, we have
$\overline{\conv(T)}=\conv(\overline T)$ (cf.~\cite{Roc1972}, Theorem 17.2) and therefore,
$(x_0,y_0)=t(x_1,y_1)+(1-t)(x_2,y_2)$, where $t\in[0,1]$,
$(x_1,y_1)\in\overline A\times\overline U$, $(x_2,y_2)\in\overline U\times\overline B$.
First observe that $t\in(0,1)$.

Indeed, suppose for instance that $(x_0,y_0)\in U\times(\overline B\cap V)$. Take an arbitrary
$x_\ast\in\intt A$ and let $r>0$, $\eps>0$ be such that
the Euclidean ball $\BB((x_\ast,y_0),r)$ is contained in $A\times V$
and $x_{\ast\ast}:=x_\ast+\eps(x_0-x_\ast)\in U$. Then
\begin{multline*}
(x_0,y_0)\in\intt(\conv(\BB((x_\ast,y_0),r)\cup\{(x_{\ast\ast},y_0)\}))\\
\subset\intt(\conv(\overline T))=\intt(\overline{\conv(T)})=\conv(T);
\end{multline*}
a contradiction.

Let $L:\RR^p\times\RR^q\too\RR$ be a linear form such that $L(x_0,y_0)=1$ and
$L\leq1$ on $T$. Since $1=L(x_0,y_0)=tL(x_1,y_1)+(1-t)L(x_2,y_2)$, we conclude that
$L(x_1,y_1)=L(x_2,y_2)=1$. Write $L(x,y)=P(x)+Q(y)$, where $P:\RR^p\too\RR$,
$Q:\RR^q\too\RR$ are linear forms.

Put $P_C:=\sup_CP$, $C\subset\RR^p$, $Q_D:=\sup_DQ$, $D\subset\RR^q$.
Since $L\leq1$ on $T$ and $L(x_1,y_1)=L(x_2,y_2)=1$, we conclude that
\begin{align*}
P_A+Q_V&=1,\\
P_U+Q_B&=1.
\end{align*}
In particular, $P_A=P_U$ iff $Q_B=Q_V$.
Consider the following two cases:

\bu $P_{A}<P_{U}$ and $Q_{B}<Q_{V}$: Then
$$
\frac{P-P_{A}}{P_{U}-P_{A}}\leq\Phi_{A,U},\quad
\frac{Q-Q_{B}}{Q_{V}-Q_{B}}\leq\Phi_{B,V}.
$$
Hence
$$
\Phi_{A,U}(x_0)+\Phi_{B,V}(y_0)\geq
\frac{P(x_0)-P_{A}}{1-Q_{B}-P_{A}}+
\frac{Q(y_0)-Q_{B}}{1-P_{A}-Q_{B}}=1.
$$

\bu $P_{A}=P_{U}$ and $Q_{B}=Q_{V}$: Then $P_{U}+Q_{V}=1$,
which implies that $(x_0,y_0)\in U\times V\subset\{L<1\}$; a contradiction.

\halfskip

Now, assume that the result is true for $N-1\geq2$. In particular,
$$
\conv(T')=\Big\{(x_1,\dots,x_{N-1})\in U_1\ttimes U_{N-1}:
\sum_{j=1}^{N-1}\Phi_{S_j,U_j}(x_j)<1\Big\}=:W'.
$$
Using (**), the case $N=2$, and Remark \ref{RemCEF}(\ref{RemCEFd}), we get
\begin{align*}
\conv(T)&=\conv((W'\times S_N)\cup((S'\times U_N))\\
&=\{(x',x_N)\in W'\times U_N: \Phi_{S', W'}(x')+\Phi_{S_N,U_N}(x_N)<1\}=W.\Dqed
\end{align*}
\end{proof}

\section{Reinhardt geometry.}\label{SectionRG}

Now we recall basic facts related to Reinhardt domains.

\begin{Definition}
We say that a set $A\subset\CC^n$ is a \textit{Reinhardt set}
\index{Reinhardt set}%
if for every $(a_1$, \dots, $a_n)\in A$ we have
$$
\{(z_1,\dots,z_n)\in\CC^n: |z_j|=|a_j|,\;j=1,\dots,n\}\subset A;
$$
cf.~\cite{JarPfl2008}, Definition 1.5.2. Put
\begin{align*}
\BVV{j}:&=\CC^{n-j-1}\times\{0\}\times\CC^{n-j},\quad \BVV{0}:=\BV_1\cup\dots\cup\BV_n,\\
\log A:&=\{(\log|z_1|,\dots,\log|z_n|): (z_1,\dots,z_n)\in A\setminus\BVV{0}\},\quad A\subset\CC^n,\\
\exp S:&=\{(z_1,\dots,z_n)\in\CC^n\setminus\BVV{0}: (\log|z_1|,\dots,\log|z_n|)\in S\},
\quad S\subset\RR^n,\\
A^\ast:&=\intt(\overline{\exp(\log A)}),\quad A\subset\CC^n.
\end{align*}
We say that a set $A\subset\CC^n$ is \textit{logarithmically convex (log-convex)}
\index{logarithmic convexity}%
if $\log A$ is convex; cf.~\cite{JarPfl2008}, Definition 1.5.5.
\end{Definition}

\begin{Theorem}[\cite{JarPfl2008}, Theorem 1.11.13]\label{1.11.13}
Let $\Omega\subset\CC^n$ be a Reinhardt domain. Then the following conditions are equivalent:

{\rm (i)} $\Omega$ is a domain of holomorphy;

{\rm (ii)} $\Omega$ is log-convex and $\Omega=\Omega^\ast\setminus
\bigcup_{\substack{j\in\{1,\dots,n\}\\ \Omega\cap\BVV{j}=\varnothing}}\BVV{j}$.
\end{Theorem}

\begin{Theorem}[\cite{JarPfl2008}, Theorem 1.12.4]\label{1.12.4}
For every Reinhardt domain $\Omega\subset\CC^n$ its envelope of holomorphy $\wdht\Omega$
is a Reinhardt domain.
\end{Theorem}

\begin{Corollary}\label{CorRGeomI}
Let $\Omega\subset\CC^n$ be a Reinhardt domain and let $\wdht\Omega$ be its envelope of holomorphy.
Then

{\rm (a)} $\BVV{j}\cap\wdht\Omega=\varnothing$ iff $\BVV{j}\cap\Omega=\varnothing$,

{\rm (b)} $\log\wdht\Omega=\conv(\log\Omega)$.

Consequently, by Theorem \ref{1.12.4},
$$
\wdht\Omega=\intt(\overline{\exp(\conv(\log\Omega))})\setminus
\bigcup_{\substack{j\in\{1,\dots,n\}\\ \Omega\cap\BVV{j}=\varnothing}}\BVV{j}=:\wdtl\Omega.
$$
\end{Corollary}

\begin{proof}
(a) If $\BVV{j}\cap\Omega=\varnothing$, then the function
$\Omega\ni z_j\tuu1/z_j$ is holomorphic on $\Omega$. Thus, it must be holomorphically continuable
to $\wdht\Omega$, which means that $\BVV{j}\cap\wdht\Omega=\varnothing$.

(b) First observe that, by Remark 1.5.6(a) from \cite{JarPfl2008}, we get
$\log\wdtl\Omega=\conv(\log\Omega)$. Consequently, $\wdtl\Omega$ is a domain of holomorphy
with $\Omega\subset\wdtl\Omega$. Hence, $\wdht\Omega\subset\wdtl\Omega$.
Finally, $\log\Omega\subset\log\wdht\Omega\subset\log\wdtl\Omega=\conv(\log\Omega)$.
\end{proof}

\begin{Proposition}[\cite{JarPfl2008}, Proposition 1.14.20]\label{1.14.20}
Let $\Omega$ be a log-convex Reinhardt domain.

{\rm (a)} Let $u\in\PSH(\Omega)$ be such that
$$
u(z_1,\dots,z_n)=u(|z_1|,\dots,|z_n|),\quad (z_1,\dots,z_n)\in\Omega.
$$
Then the function
$$
\log\Omega\ni(x_1,\dots,x_n)\overset{\phi}\tuu u(e^{x_1},\dots,e^{x_n})
$$
is convex.

{\rm (b)} Let $\phi\in\CVX(\log\Omega)$. Then the function
$$
\Omega\setminus\BVV{0}\ni z\overset{u}\tuu\phi(\log|z_1|,\dots,\log|z_n|)
$$
is plurisubharmonic.
\end{Proposition}

\begin{Corollary}\label{CorRGeomII}
Let $\varnothing\neq A\subset\Omega$, where $\Omega$ is log-convex
Reinhardt domain and $A$ is an Reinhardt open set. Then
$$
h_{A,D}^\ast(z)=\Phi_{\log A,\log\Omega}(\log|z_1|,\dots,\log|z_n|),\quad z=(z_1,\dots,z_n)\in\Omega\setminus\BVV{0};
$$
cf.~Definition \ref{DefCEF}.
\end{Corollary}

\begin{proof}
Since $A$ and $\Omega$ are invariant under rotations, we easily conclude that
$$
h_{A,D}^\ast(z)=h_{A,D}^\ast(|z_1|,\dots,|z_n|),\quad z=(z_1,\dots,z_n)\in\Omega.
$$
Thus, by Proposition \ref{1.14.20},
$$
h_{A,D}^\ast(z)=\phi(\log|z_1|,\dots,\log|z_n|),\quad
z=(z_1,\dots,z_n)\in\Omega\setminus\BVV{0},
$$
where $\phi\in\CVX(\log\Omega)$.
Clearly, $h_{A,D}^\ast=0$ on $A$. Thus $\phi=0$ on $\log A$. Finally,
$\phi\leq\Phi_{\log A,\log\Omega}$.

To prove the opposite inequality, observe that by Proposition \ref{1.14.20}, the function
$$
\Omega\setminus\BVV{0}\ni z\overset{u}\tuu\Phi_{\log A,\log\Omega}(\log|z_1|,\dots,\log|z_n|)
$$
is plurisubharmonic, $u<1$, and $u=0$ on $A\setminus\BVV{0}$. Consequently,
$u$ extends to a $\wdtl u\in\PSH(\Omega)$. Clearly, $\wdtl u\leq1$ and $\wdtl u=0$ on $A$.
Thus $\wdtl u\leq h_{A,D}^\ast$.
\end{proof}

\section{Proof of the cross theorem (Theorem \ref{ThmMainCross}) in the case
where $D_j$ is a Reinhardt domain of holomorphy and $A_j$ is an open Reinhardt set,
$j=1,\dots,N$.}\label{SectionProof}

We have to prove
that the envelope of holomorphy $\wdht{\BX}$ of the domain $\BX$ coincides with
$$
\wdtl{\BX}:=\Big\{(z_1,\dots,z_N)\in D_1\ttimes D_N: \sum_{j=1}^Nh_{A_j,D_j}^\ast(z_j)<1\Big\}.
$$
First, observe that $\wdtl{\BX}$ is a domain of holomorphy containing $\BX$.
Thus $\wdht{\BX}\subset\wdtl{\BX}$. On the other hand, by Proposition \ref{PropCEF} and
Corollary \ref{CorRGeomII}, $\log\wdtl{\BX}=\conv(\log\BX)=\log\wdht{\BX}$.
Thus, using Corollary \ref{CorRGeomI}, we only need to show that if
$\BVV{j}\cap\wdtl{\BX}\neq\varnothing$,
then $\BVV{j}\cap\BX\neq\varnothing$. Indeed, let for example
$a=(a_1,\dots,a_N)\in\BVV{n}\cap\wdtl{\BX}\neq\varnothing$. Take arbitrary
$b_j\in A_j$, $j=1,\dots,N-1$. Then $(b_1,\dots,b_{N-1},a_N)\in\BVV{n}\cap\BX$.
\qed


\bibliographystyle{amsplain}

\end{document}